%% file: dk_eq_quillen_eq.tex
\pdfoutput=1

\documentclass[10pt,a4paper,DIV=14]{scrartcl}

\usepackage[utf8]{inputenc}
\usepackage[T1]{fontenc}

\usepackage{hyperref}

\hypersetup%
{%
	colorlinks,
	pdfpagelabels,
	urlcolor=red,
	linkcolor=teal,
	citecolor=orange,
	unicode,
	hyperindex=true,
	final=true,
	pdfauthor=Alexander Körschgen,
	pdftitle={Dwyer-Kan Equivalences Induce Equivalences on
		Topologically Enriched Presheaves},
	pdfsubject={2010 AMS MSC: 18D20},
	pdfkeywords=
	{%
		Dwyer-Kan Equivalence,
		Model Category,
		Topologically Enriched Category,
		Enriched Presheaf
	}
}

\usepackage{mystyle-minimal-koma}
\usepackage{mystyle-operators}
\usepackage{mystyle-environments-numbered}
\usepackage{mydiagrams}

\usepackage[nottoc]{tocbibind}

\usepackage{nicefrac}
\usepackage{stmaryrd} %
\usepackage{mathtools} %
\usepackage{tensor}

\usepackage{pdfsync}

\usepackage{bookmark}

\input{preamble_operators}

\deffootnote[1em]{1em}{1em}{\textsuperscript{\thefootnotemark}}

\title{Dwyer-Kan Equivalences Induce Equivalences on
	Topologically Enriched Presheaves}
\author%
{%
	Alexander Körschgen%
	\thanks
	{%
		Mathematisches Institut der Universität Bonn,
		Endenicher Allee 60, 53115 Bonn, Germany.\newline
		E-mail address:
			\href{mailto:alex@math.uni-bonn.de}{alex@math.uni-bonn.de}
	}
}
\date{}
\begin{document}%
\maketitle%

\input{abstract}
\pdfbookmark[section]{\contentsname}{toc}
\tableofcontents
\input{sec_intro}
\input{sec_basic}

\input{sec_induced_eq}
\appendix
\section{Appendix}
\input{subsec_app_cgwh}

\input{subsec_app_closed_incl}

\input{subsec_app_cofibrant_objects}

\bibliographystyle{abbrv}
\bibliography{dk_eq_quillen_eq}

\end{document}

%% file: preamble_operators.tex
\newcommand{\catc}{\mathcal{C}}
\newcommand{\catd}{\mathcal{D}}
\DeclareMathOperator{\ev}{ev}
\DeclareMathOperator{\cof}{cof}
\newcommand{\cN}{\mathcal{N}}
\newcommand{\cM}{\mathcal{M}}
\DeclareMathOperator{\cf}{cf}

%% file: abstract.tex
\begin{abstract}
This brief note elaborates on a result by Gepner and Henriques. They have
shown that a Dwyer-Kan equivalence between two small, topological categories
gives rise to a Quillen equivalence of the associated categories of
topologically enriched presheaves. We present a more detailed account of
their proof.
\end{abstract}

%% file: sec_intro.tex
\section{Introduction}
\label{sec:intro}
In~\cite[Lemma~A.6]{gepner_homotopy_2007}, Gepner and Henriques show that,
given a Dwyer-Kan equivalence $f\colon \catc \to \catd$ between topologically
enriched index categories, the associated functor $f^\ast\colon
\Pre(\catd,\Top) \to \Pre(\catc,\Top)$ between the respective categories of
topologically enriched presheaves is the right adjoint of a Quillen
equivalence $f_! \dashv f^\ast$. We give a more detailed account of their
proof, discussing the required results on point-set topology as well as the
necessary transfinite techniques in depth.

One might be tempted to apply~\cite[Proposition 2.4]{guillou_enriched_2011} to
deduce the desired result. However, a certain assumption of this proposition
does not hold in our context. We will elaborate on this problem in
Remark~\ref{rem:cubeproblem}.

Section~\ref{sec:basic} introduces basic properties of the category
$\Pre(\catc,\Top)$ of topologically enriched presheaves on a small,
topologically enriched category $\catc$. The following
Section~\ref{sec:induced_eq} recalls the definition of a Dwyer-Kan equivalence
and proceeds by reproducing the proof of Gepner and Henriques, which we
divided into two lemmas and the final Theorem~\ref{thm:dkeqinducesquilleneq}.

By \emph{space}, we mean a compactly generated weak Hausdorff space. The
required point-set foundations are spelled out in the first two subsections of
the appendix while the last subsection of the appendix yields a helpful
characterization of cofibrant objects of $\Pre(\catc,\Top)$.

These notes used to be a part of~\cite{koerschgen_comparison_2016} until the
author decided to split them off.

%% file: sec_basic.tex
\section{Basic Properties of \texorpdfstring{$\Pre(\catc,\Top)$}{Pre(C,Top)}}
\label{sec:basic}

\begin{definition}
	Let $\catc$ be a topologically enriched category. 
	\begin{enumerate}
		\item Denote by $\Pre(\catc,\Top)$ the category of enriched functors
			$\catc \to \Top$. In this paper, it will always be equipped with the
			projective model structure.
		\item The category $\Pre(\catc,\Top)$ is bitensored over $\Top$. For $X
			\in \Pre(\catc,\Top), K \in \Top$, and $c \in \catc$, we have
			\[
				(X \otimes K)(c) = X(c) \times K, \quad (X^K)(c) = X(c)^K.
			\]
		\item For every $c \in \catc$, write $\ev_c: \Pre(\catc,\Top)
			\to \Top$ for the functor given by evaluation at $c$.
	\end{enumerate}
\end{definition}

\begin{prop}
	\label{prop:toppresheavesgen}
	Let $\catc$ be a topologically enriched category.
	\begin{enumerate}
		\item For $c \in \catc$, the functor $\ev_c$ has a left adjoint $F_c:
			\Top \to \Pre(\catc,\Top)$, given by $(F_c K)(c^\prime) =
			\catc(c^\prime,c) \times K$. The pair $F_c \dashv \ev_c$ is
			Quillen.
		\item The projective model structure on $\Pre(\catc,\Top)$ is
			cofibrantly generated by the generating cofibrations
			\[
				I_\catc := \{F_c (S^{n-1} \hookrightarrow D^n)\sth
					n \geq 0, c \in \catc\}
			\]
			and the generating trivial cofibrations
			\[
				J_\catc := \{F_c (D^n \hookrightarrow D^n \times [0;1])
					\sth n \geq 0, c \in \catc\}.
			\]
			The domains and codomains of the maps in $I_\catc \cup J_\catc$
			are cofibrant.
		\item $\Pre(\catc,\Top)$ is a topological model category. In
			particular, we have natural homeomorphisms
			\[
				\Pre(\catc,\Top) (X \otimes K,X^\prime) \cong
				\Pre(\catc,\Top)(X,(X^\prime)^K)
			\]
			and the pushout-product axiom holds.
		\item Any cofibration in $\Pre(\catc,\Top)$ is a levelwise closed
			cofibration, i.e., a Hurewicz cofibration with closed image.
	\end{enumerate}
\end{prop}
\begin{proof}
It is an easy exercise to verify that $F_c$ is left adjoint to $\ev_c$. As
$\ev_c$ preserves fibrations and trivial fibrations, the pair $F_c \dashv
\ev_c$ is Quillen, showing (i).

The category of (CGWH) topological spaces satisfies the monoid axiom by
\cite[Lemma 2.3]{hovey_monoidal_1998}. Therefore, we may apply \cite[Theorem
24.4]{shulman_homotopy_2006} and deduce (ii). Part (iii) is a
straight-forward computation.

Part (iv) uses the theory of \emph{h-cofibrations}. see, e.g.,
\cite[Definition A.1.18]{schwede_global_2016}.  Every topological
space is fibrant, so every object of $\Pre(\catc,\Top)$ is fibrant. From
\cite[Corollary A.1.20.(iii)]{schwede_global_2016}, we deduce
that every cofibration in $\Pre(\catc,\Top)$ is a h-cofibration. Picking $c
\in \catc$, the functor $\ev_c: \Pre(\catc,\Top)$ commutes with colimits and
tensors.  Thus, it takes h-cofibrations to cofibrations by part (ii) of the
same Corollary. So, every cofibration is a levelwise h-cofibration of
topological spaces. An inspection of the definitions shows that h-cofibrations
in $\Top$ are exactly the Hurewicz cofibrations.

To conclude the proof of part (iv), it remains to show that a cofibration in
$\Pre(\catc,\Top)$ is levelwise a closed inclusion. This is true for the
generating cofibrations by their description above. The characterization of
cofibrations, see Subsection~\ref{subsec:cofobjects}, together with
Lemma~\ref{lem:closedinclusionsclosed} shows that it is also true for an
arbitrary cofibration.
\end{proof}

%% file: sec_induced_eq.tex
\section{Dwyer-Kan Equivalences and Induced Equivalences}

Let us begin by recalling the definition of a Dwyer-Kan equivalence:

\begin{definition}
	Let $\catc$ be topologically enriched category. Then the ordinary
	category $\pi_0 \catc$ has $\ob \pi_0 \catc = \ob \catc$ and
	$(\pi_0 \catc)(c,c^\prime) = \pi_0 (\catc(c,c^\prime))$ with composition
	defined in the obvious way.

	An enriched functor $f: \catc \to \catd$ of topologically enriched
	categories is a \emph{Dwyer-Kan equivalence} if the induced functor $\pi_0
	f: \pi_0 \catc \to \pi_0 \catd$ is an equivalence of categories and $f$ is
	a weak equivalence on all mapping spaces.
\end{definition}

Next, we establish the Quillen pair between topologically enriched presheaf
categories induced by an enriched functor. For a Dwyer-Kan equivalence, this
Quillen pair is a Quillen equivalence as we will have seen at the end of this
section.

\label{sec:induced_eq}
\begin{prop}
	Let $f: \catc \to \catd$ be an enriched functor between topologically
	enriched categories. Then the restriction functor $f^\ast:
	\Pre(\catd,\Top) \to \Pre(\catc,\Top)$ has a left adjoint $f_!$.
	The pair $f_! \dashv f^\ast$ is Quillen. Furthermore, both functors
	respect the tensoring over $\Top$. Finally, both functors preserve
	colimits.
	\label{prop:quillenpairgenprop}
\end{prop}
\begin{proof}
	The construction of $f_!$ can be found in \cite[Lemma
	A.6]{gepner_homotopy_2007}. Alternatively, one can use the left Kan
	extension of the functor
	\[
		\begin{array}{ccc}
			\catc & \to & \Pre(\catd,\Top) \\
			c & \mapsto & F_{f(c)}(\ast)
		\end{array}
	\]
	along the enriched Yoneda embedding $F_{-} (\ast): \catc \to
	\Pre(\catc,\Top)$.
	To show that the pair $f_! \dashv f^\ast$ is
	Quillen, it suffices to observe that $f^\ast$ clearly preserves fibrations
	and trivial fibrations.

	The left adjoint $f_!$ preserves colimits, and $f^\ast$ preserves
	colimits, too, because they are computed levelwise.  Moreover, it is
	evident from the definitions that $f^\ast$ preserves the tensoring.  To
	show that $f_!$ preserves the tensoring, observe that the cotensoring is
	also compatible with $f^\ast$.  Let $X \in \Pre(\catc,\Top), Y \in
	\Pre(\catd,\Top)$, and $K \in \Top$. We compute
	\begin{align*}
		& \Pre(\catd,\Top)(f_! (X \otimes K),Y) \\
		\cong & \Pre(\catc,\Top)(X \otimes K,f^\ast Y) \\
		\cong & \Pre(\catc,\Top)(X,(f^\ast Y)^K) \\
		\cong & \Pre(\catc,\Top)(X,f^\ast(Y^K)) \\
		\cong & \Pre(\catd,\Top)(f_! X, Y^K) \\
		\cong & \Pre(\catd,\Top)((f_! X) \otimes K,Y)
	\end{align*}
	naturally. Hence, there is a natural isomorphism $f_! (X \otimes K) \cong
	(f_! X) \otimes K$.

	\end{proof}

The following two lemmas provide the technical details necessary for the proof
of the main theorem at the end of this subsection.

\begin{lemma}
	Let $f: \catc \to \catd$ be as before. If the functor $f$ is homotopically
	fully faithful, then the unit $\eta_X: X \to f^\ast f_! X$ is a level weak
	equivalence in $\Pre(\catc,\Top)$ for all cofibrant $X \in \catc$.
	\label{lem:unitwe}
\end{lemma}
\begin{proof}
	Let $\cN$ be the class of all objects of $\Pre(\catc,\Top)$ for
	which the unit is a weak equivalence. Our goal is to show that the
	conditions of Lemma~\ref{lem:classcontcof} are satisfied.
	
	The only map whose codomain is the empty presheaf $\emptyset$ is its
	identity $\id_\emptyset$, which is an isomorphism. We proceed to verify
	(i)-(iii).
	\begin{enumerate}
		\item If $X \in \cN$ and $X^\prime$ is a retract of $X$, then the unit map
			$\eta_{X^\prime}$ is a retract of $\eta_X$ by naturality of the
			unit. A retract of a weak equivalence is a weak equivalence.
			Hence, $\eta_{X^\prime}$ is a weak equivalence.
		\item First, let us show that $\eta_X$ is a weak equivalence whenever
			$X$ is the domain or codomain of a map in $I$. These domains and
			codomains are of the form $F_c (K) = F_c (\ast) \otimes K$ for $c
			\in \catc$ and $K$ a space. The evaluation of $\eta_{F_c (K)}$ at
			some $c^\prime \in \catc$ fits into a diagram

			\[
				\begin{tikzcd}
					(F_c (\ast) \otimes K) (c^\prime) \dar[equals]
						\rar{\eta_{F_c (K)}} &
					(f^\ast f_! (F_c (\ast) \otimes K)) (c^\prime) \dar
						[equals] \\
					\catc (c^\prime,c) \times K 
						\arrow[rddd, "f_{c^\prime,c} \times K"', bend
							right,end anchor=west] &
					(f_! (F_c (\ast) \otimes K)) (f(c^\prime)) \dar{\cong} \\
					& (f_! (F_c (\ast)) \otimes K) (f(c^\prime)) \dar{\cong} \\
					& (F_{f(c)} (\ast) \otimes K) (f(c^\prime)) \dar[equals] \\
					& \catd(f(c^\prime),f(c)) \times K
				\end{tikzcd}
			\]
			using the compatibility of $f_!$ with tensors and the formula $f_!
			(F_c (\ast)) \cong F_{f(c)} (\ast)$ from the proof of
			Proposition~\ref{prop:quillenpairgenprop}. The bent map is a weak
			equivalence because $f_{c^\prime,c}: \catc(c^\prime,c) \to
			\catd(f(c^\prime),f(c))$ is a weak equivalence since $f$ is
			homotopically fully faithful by assumption. Hence, $\eta_{F_c (K)}$
			is a level weak equivalence.

			Given a pushout diagram
			\[
				\pushout{A}{X}{B}{X^\prime}{}{I \ni}{}{}
			\]
			with $X \in \cN$, the square 
			\[
				\pushout{f^\ast f_! A}{f^\ast f_! X}{f^\ast f_! B}%
					{f^\ast f_! X^\prime}{}{}{}{}
			\]
			is pushout by Proposition~\ref{prop:quillenpairgenprop}.
			Evaluating at some $c \in \catc$, we obtain a commutative cube
			\[
				\begin{tikzcd}
					& (f^\ast f_! A) (c) \arrow[rr] \arrow[dd,rightarrowtail]
					&& (f^\ast f_! X) (c) \arrow[dd] \\
					A(c) \arrow[rr, crossing over] \arrow[dd,rightarrowtail]
						\arrow[ur,"\simeq"]
					&& X(c) \arrow[ur,"\simeq"'] \\
					& (f^\ast f_! B)(c) \arrow[rr] && (f^\ast f_! X^\prime)(c) \\
					B(c) \arrow[rr] \arrow[ur,"\simeq"]
					&& X^\prime (c) \arrow[ur,"\Rightarrow \simeq"']
						\arrow[from=uu, crossing over]
				\end{tikzcd}
			\]
			with both the front and rear faces being pushout. The maps
			perpendicular to the plane of projection are the respective unit
			maps evaluated at $c$. Three of these are weak
			equivalences, namely $\eta_X(c)$ by assumption and $\eta_A(c)$
			and $\eta_B(c)$ by our considerations from before.
			
			The map $A \to B$ is a cofibration, so $A(c) \to B(c)$ is a
			Hurewicz cofibration by
			Proposition~\ref{prop:toppresheavesgen}.(iv). The left adjoint
			$f_!$ preserves cofibrations, so $f_! A \to f_! B$ is a
			cofibration and 
			\[
				(f^\ast f_! A)(c) = (f_! A)(f(c)) \to
					(f_! B)(f(c)) = (f^\ast f_! B)(c)
			\]
			is a Hurewicz cofibration. In this situation, the map
			$\eta_{X^\prime}(c)$ between the pushouts is a weak equivalence as
			well (\cite[Proposition 4.8 (b)]{boardman_homotopy_1973} or
			\cite[Proposition 1.1]{malkiewich_homotopy_2014}).  Since $c$ was
			arbitrary, $\eta_{X^\prime}$ is a level weak equivalence.
		\item Assume that there is a $\lambda$-sequence $X: \lambda \to
			\Pre(\catc,\Top)$ such that (a)-(c) from
			Lemma~\ref{lem:classcontcof} hold. A case distinction is in order.
			\begin{description}
				\item[$\lambda$ is $0$]	
					A $0$-sequence is just an empty diagram with colimit the
					initial object $\emptyset$. Also, $f^\ast f_! \emptyset =
					\emptyset$ and the map $\eta_\emptyset$ must be
					$\id_\emptyset$ which is a weak equivalence.
				\item[$\lambda$ a successor ordinal]
					If $\lambda = \mu + 1$, we have $\colim_{\beta < \lambda}
					X_\beta = X_\mu \in \cN$ by assumption.
				\item[$\lambda$ a limit ordinal]
					Some categorical yoga shows that
					$\colim_{\beta < \lambda} \eta_{X_\beta} =
					\eta_{\colim_{\beta < \lambda} X_\beta}$. Therefore, it
					is sufficient to show that $\colim \eta_{X_\beta}$ is a
					weak equivalence.
				
					By condition (b), the maps $X_\beta \to X_{\beta + 1}$ are
					cofibrations, hence levelwise closed inclusions by
					Proposition~\ref{prop:toppresheavesgen}.  The maps $f_!
					X_\beta \to f_!  X_{\beta +1}$ are cofibrations and
					levelwise closed inclusions as well, and we derive that
					$f^\ast f_! X_\beta \to f^\ast f_!  X_{\beta + 1}$ are
					levelwise closed inclusions. Hence, both $\colim
					X_\beta (c)$ and $\colim f^\ast f_!  X_\beta(c)$ are taken
					along closed inclusions, and
					the maps $X_\beta \to f^\ast f_! X_\beta$ are weak
					equivalences because $X_\beta \in \cN$ by assumption (c):
					\[
						\begin{tikzcd}[column sep=large]
							X_\beta (c) \rar{\simeq} \dar[rightarrowtail] &
							(f^\ast f_! X_\beta)(c) \dar[rightarrowtail] \\
							X_{\beta + 1 } (c) \rar{\simeq} \dar[dotted] &
							(f^\ast f_! X_{\beta + 1})(c) \dar[dotted] \\
							\colim X_\beta (c) \rar[swap]{\colim \eta_{X_\beta}} &
							\colim (f^\ast f_! X_\beta) (c)
						\end{tikzcd}
					\]

					By Lemma~\ref{lem:lambdaseqclosedinclusion}, $\pi_k \colim
					X_\beta(c) \cong \colim \pi_k X_\beta (c)$ and $\pi_k
					\colim (f^\ast f_!  X_\beta)(c) \cong \colim \pi_k (f^\ast
					f_! X_\beta)(c)$. Thus, $\pi_k (\colim \eta_{X_\beta})$ is
					a colimit of isomorphisms and therefore an isomorphism
					itself. In conclusion, $\eta_{\colim X_\beta} = \colim
					\eta_{X_\beta}$ is a weak equivalence as desired. Note
					that we did not need condition~(a) in our specific
					situation.
			\end{description}
	\end{enumerate}
\end{proof}

\begin{lemma}
	Let $\catd$ be a topologically enriched category and $Y \in
	\Pre(\catd,\Top)$. If $a: d \to d^\prime$ is a homotopy
	equivalence in $\catd$, then so is $Y(a): Y(d^\prime) \to Y(d)$.
	\label{lem:homeqtohomeq}
\end{lemma}
\begin{proof}
	If $a$ is such a homotopy equivalence, then there is $b: d^\prime \to d
	\in \catd$ such that $[b \circ a] = [\id_d] \in \pi_0 (\catd(d,d))$ and
	$[a \circ b] = [\id_{d^\prime}] \in \pi_0 (\catd(d^\prime,d^\prime))$.

	The
	enriched functor $Y$ induces a map
	\[
		\pi_0 (\catd(d,d)) \to \pi_0 (\Top(Y(d),Y(d))).
	\]
	As $[b \circ a]$ and $[\id_d]$ are the same on the left hand side, we
	obtain that $[Y(\id_d)] = [\id_{Y(d)}]$ and $[Y(b \circ a)] = [Y(a) \circ
	Y(b)]$ agree, too. Two maps represent the same path component in
	$\pi_0(\Top(Y(d),Y(d)))$ if and only if they are homotopic. 
	Therefore, $Y(a) \circ Y(b) \simeq \id_{Y(d)}$. By
	an analogous argument, $Y(b) \circ Y(a) \simeq \id_{Y(d^\prime)}$.
	We conclude that $Y(a)$ is a homotopy equivalence.
\end{proof}

\begin{thm}
	\label{thm:dkeqinducesquilleneq}
	Let $f: \catc \to \catd$ be a Dwyer-Kan equivalence between topologically
	enriched categories. Then the induced Quillen pair $f_! \dashv f^\ast$
	is a Quillen equivalence.
\end{thm}
\begin{proof}
	Let $X \in \Pre(\catc,\Top)$ be cofibrant, $Y \in \Pre(\catd,\Top)$, and
	$\alpha: X \to f^\ast Y$ be a map with adjoint $\beta: f_! X \to Y$. By
	one of the many equivalent characterizations of a Quillen equivalence, we
	need to show that $\alpha$ is a (level) weak equivalence if and only
	$\beta$ is (note that any $Y$ is fibrant).  
	The map $\alpha$ factors through the unit $\eta_X$, which is a weak
	equivalence by Lemma~\ref{lem:unitwe}:
	\[
		\begin{tikzcd}
			& f^\ast f_! X \drar{f^\ast(\beta)} & \\
			X \arrow[ur,"\simeq"', "\eta_X"] \arrow[rr,"\alpha"'] && f^\ast Y
		\end{tikzcd}
	\]
	Therefore, $\alpha$ is a weak equivalence if and only if $f^\ast (\beta)$
	is. It remains to show that this is the case if and only if $\beta$ is a
	weak equivalence.
	
	Assume that $f^\ast(\beta)$ is a weak equivalence and choose $d \in \catd$.
	As $f$ is homotopically essentially surjective, there is $c \in \catc$ and
	a map $a: fc \to d$ in $\catd$ that is a homotopy equivalence. We obtain a
	commutative diagram
	\[
		\begin{tikzcd}[column sep=huge]
			(f_! X) (fc) = (f^\ast f_! X)(c)
				\arrow[r,"\simeq"',"\beta_{fc} = (f^\ast (\beta))_c"]&
			Y (fc) = (f^\ast Y)(c) \\
			(f_! X) (d) \uar{a^\ast} \rar[swap]{\beta_d}&
			Y (d) \uar[swap]{a^\ast} 
		\end{tikzcd}
	\]
	The vertical arrows are homotopy equivalences by
	Lemma~\ref{lem:homeqtohomeq}, therefore, $\beta_d$ is a weak equivalence
	and $\beta$ is a level weak equivalence.

	Vice versa, let $\beta$ be a level weak equivalence. Obviously, this
	implies that $f^\ast (\beta)$ is a level weak equivalence, concluding the
	proof.
\end{proof}

This concludes the main proof of this note. Let us end this section with a
remark promised in the introduction.

\begin{rem}
	\label{rem:cubeproblem}
	As already mentioned in the introduction, there is an issue preventing us
	from applying~\cite[Proposition 2.4]{guillou_enriched_2011} to deduce the
	previous theorem. Namely, it does not hold in general that cofibrations in
	$\Pre(\catc,\Top)$ are level cofibrations of topological spaces.

	To circumvent this issue, we have to work with Hurewicz cofibrations
	in Lemma~\ref{lem:unitwe}, which still interact nicely with weak equivalences of
	topological spaces under pushouts and transfinite composition. The
	remaining parts of the proof are independent of this issue.

	One of the assumptions of \cite[Proposition 2.4]{guillou_enriched_2011},
	hidden in \cite[Theorem 4.31]{guillou_enriched_2011}, is that the functors
	$\catc(c,c^\prime) \otimes -$ preserve cofibrations. If this is the case,
	then cofibrations in $\Pre(\catc,\Top)$ are indeed level cofibrations, and
	the arguments in Lemma~\ref{lem:unitwe} become much simpler.
	However, it is usually not the case that $\catc(c,c^\prime) \otimes -$
	preserves cofibrations unless the $\catc(c,c^\prime)$ happen to be
	cofibrant themselves.
\end{rem}

%% file: subsec_app_cgwh.tex
\subsection{Compactly Generated Weak Hausdorff Spaces}
\label{subsec:cgwhspaces}
The%
\footnote%
{%
	This subsection and the following one have been copied
	from~\cite{koerschgen_comparison_2016} as of version 2.
}
main body of this paper takes place in the category of compactly generated
weak Hausdorff spaces, also referred to CGWH spaces. Before we deal with the
necessary point-set arguments, let us make the used terminology precise.

A space is \emph{compact} if every open cover admits a finite subcover. This
is also being referred to as \emph{quasi-compact} in other sources which
include the Hausdorff property into the definition of compactness.

Moreover, a space $X$ is \emph{compactly generated} if, for any subset $Y \subseteq
X$, $Y$ is closed if and only if $u^{-1} (Y)$ is closed for every compact
Hausdorff $K$ and every continuous $u: K \to X$. The space $X$ is weak
Hausdorff if for every such $u$ and $K$, the image $u(K)$ is closed
in $X$.

These definitions are taken from~\cite{strickland_category_2009}. Note that
this terminology varies within the literature, and some sources refer to
compactly generated spaces as \emph{$k$-spaces} while they take compactly generated
spaces to be compactly generated weak Hausdorff spaces in our sense.

Note that the property of being compactly generated is a local property, i.e.,
a space is compactly generated if and only if each point has a compactly
generated neighborhood. The property of being weak Hausdorff is not local,
though.

In this paper, we refer to CGWH spaces as \emph{(topological) spaces} and denote
the corresponding category by $\Top$. Within the next subsections, we will
have to deal with their point-set subtleties and cite statements about not
necessarily CGWH spaces. To this end, we will use the term \emph{general
topological space} for a space that is not necessarily CGWH.

The category of CGWH spaces is cocomplete. Limits and colimits may, however,
differ from those computed in the category of general topological spaces. Our
convention is that limits and colimits are computed in CGWH unless it is
explicitly declared that the diagram in consideration lives in the category
of general topological spaces. In this case, limits and colimits are to be
taken in the category of general topological spaces. The latter situation does
only occur in Subsection~\ref{subsec:closedincl}.

For the special case of products, we adopt the following notation
from~\cite{strickland_category_2009}: Given two spaces $X$ and $Y$, we denote
by $X \times_0 Y$ the product taken in the category of general topological
spaces, which is not necessarily compactly generated. In contrast, $X \times
Y$ shall denote the product in the category of CGWH spaces.

%% file: subsec_app_closed_incl.tex
\subsection{Closed Inclusions and CGWH Colimits}
\label{subsec:closedincl}

We will now shed some light on situations where specific colimits
agree regardless of whether they are computed in CGWH or in the category of
general topological spaces.

\begin{lemma}[{\cite[Section 2.4, p.~59]{hovey_model_1999}}]
	\label{lem:gentopspacessomecolimitsagree}
	The category of topological spaces is cocomplete. In the case of pushouts
	along closed inclusions or transfinite compositions of injections,
	colimits may be computed in the category of general topological spaces
	since they are already weak Hausdorff.
\end{lemma}

\begin{lemma}
	\label{lem:closedinclusionsclosed}
	In the category of topological spaces, closed inclusions are closed under
	pushouts, transfinite compositions, and retracts.
\end{lemma}
\begin{proof}
	As weak Hausdorff spaces are automatically $T_1$, a closed inclusion in the
	category of topological spaces is a closed $T_1$ inclusion in the category of
	general topological spaces. Retracts of maps of topological spaces are also
	retracts of maps of general topological spaces. Also, the relevant pushouts and
	transfinite compositions can be computed in the category of general
	topological spaces.

	The claim follows from the proof of \cite[Lemma 2.4.5]{hovey_model_1999}
	for the cases of pushouts and transfinite compositions and from the proof
	of \cite[Corollary 2.4.6]{hovey_model_1999} for the case of retracts.
\end{proof}

Let us end this subsection by noting that sequential colimits along closed
inclusions commute with $\pi_k$.

\begin{lemma}
	\label{lem:lambdaseqclosedinclusion}
	Let $\lambda$ be a limit ordinal and $X: \lambda \to \Top$ be a
	$\lambda$-sequence along closed inclusions. Furthermore, let $K$ be a
	compact space.  Then any map $K \to \colim_{\beta < \lambda} X_\beta$
	factors through some $X_\mu$, $\mu < \lambda$. In particular, the
	canonical map
	\[
		\colim_{\beta < \lambda} \pi_k (X_\beta) \to
		\pi_k (\colim_{\beta < \lambda} X_\beta)
	\]
	is an isomorphism.
\end{lemma}
\begin{proof}
	The colimit can be computed in the category of general topological spaces.
	As in the proof of the previous lemma, a closed inclusion is a closed
	$T_1$ inclusion of general topological spaces. Also, note that the
	cofinality $\cf \lambda$ of $\lambda$ is infinite because $\lambda$ is a
	limit ordinal. In particular, $\lambda$ is $\gamma$-filtered for each
	finite cardinal~$\gamma$ in the sense of \cite[Definition
	2.1.2]{hovey_model_1999}. Therefore, \cite[Proposition
	2.4.2]{hovey_model_1999} tells us that the canonical map
	\[
		\colim_{\beta < \lambda} \Top(K,X_\beta) \to
		\Top(K,\colim_{\beta < \lambda} X_\beta)
	\]
	is an isomorphism, proving the first claim. The second claim follows by a
	standard argument.
\end{proof}

%% file: subsec_app_cofibrant_objects.tex
\subsection{Cofibrant Objects in Cofibrantly Generated Model Categories}
\label{subsec:cofobjects}

Recall that cofibrations in a cofibrantly generated model category $\cM$ with
generating cofibrations $I$ are precisely the retracts of transfinite
compositions of pushouts of elements of $I$ \cite[Corollary
10.5.22]{hirschhorn_model_2003}. An additional assumption, which is
satisfied by $\Pre(\catc,\Top)$, allows us to simplify the characterization of
cofibrant objects. It shall be noted that this extra assumption does usually
not apply in pointed contexts.

\begin{lemma}
	\label{lem:classcontcof}
	Let $\cM$ be as above and suppose that any map $X \to \emptyset$ is
	an isomorphism. Let $\cN$ be a
	class of objects of $\cM$. Assume that $\cN$ satisfies the following
	properties:
	\begin{enumerate}
		\item $\cN$ is closed under retracts.
		\item If $X \in \cN$ and $X^\prime$ is a pushout of $X$ along a
			generating cofibration,
			\[
				\pushout{A}{X}{B}{X^\prime}{}{I \ni}{}{}
			\]
			then $X^\prime \in \cN$.
		\item If $\lambda$ is an ordinal and $X: \lambda \to \cM$
			is a $\lambda$-sequence such that
			\begin{enumerate}
				\item $X_0 = \emptyset$ if $\lambda > 0$, 
				\item $X_\beta \in \cN$ for $\beta < \lambda$, and
				\item for $\beta < \lambda$ such that $\beta + 1 < \lambda$,
					there is a pushout 
					\[
						\pushout{A_\beta}{X_\beta}{B_\beta}{X_{\beta + 1}}{}{I \ni}{}{}
					\]
			\end{enumerate}
			then $\colim_{\beta < \lambda} X_\beta \in \cN$.
	\end{enumerate}
	Then $\cN$ contains all cofibrant objects. Moreover, the class of
	cofibrant objects is the smallest class satisfying these properties.
\end{lemma}
\begin{proof}
	First, note that $\emptyset \in \cN$ because the empty $0$-sequence, whose
	colimit is $\emptyset$, satisfies all the conditions~(iii).(a)-(c).

	Let $X^\prime \in \cM$ be cofibrant, i.e., the map $\emptyset \to
	X^\prime$ is a cofibration. We wish to show that $X^\prime \in \cN$. As
	$\cM$ is cofibrantly generated, this means that there is an ordinal
	$\lambda$ and a $\lambda$-sequence $X: \lambda \to \cM$ such that
	\begin{itemize}
		\item there is a retract diagram
			\[
				\begin{tikzcd}
					\emptyset \rar \dar \arrow[rr,bend left,"\id"] & 
					X_0 \dar \rar & \emptyset \dar \\
					X^\prime \rar \arrow[rr,bend right,"\id"'] &
					\colim_{\beta < \lambda} X_\beta \rar & X^\prime
				\end{tikzcd}
			\]
		\item for $\beta < \lambda$ such that $\beta + 1 < \lambda$, there
			is a pushout 
			\[
				\pushout{A_\beta}{X_\beta}{B_\beta}{X_{\beta + 1}}{}{I \ni}{}{}
			\]
	\end{itemize}
	As the map $X_0 \to \emptyset$ is an isomorphism by assumption, we can
	assume without loss of generality that $X_0 = \emptyset$. Furthermore,
	$\cN$ is closed under retracts, therefore, it suffices to show that 
	$\colim_{\beta < \lambda} X_\beta \in \cN$.

	Conditions (iii).(a) and (iii).(c) hold for the $\lambda$-sequence $X$,
	and we wish to verify the remaining condition~(iii).(b), i.e., $X_\beta
	\in \cN$ for $\beta < \lambda$.  If this is not the case, let $\mu <
	\lambda$ be the smallest ordinal such that $X_\mu \notin \cN$. Since
	$\emptyset \in \cN$ and $X_0 = \emptyset$, we must have $\mu > 0$.  The
	truncated diagram $X_{|\mu}$ is a $\mu$-sequence and satisfies
	(iii).(a)-(c). Hence, $\colim_{\beta < \mu} X_\beta \in \cN$.
	
	If $\mu = \nu+1$ is a successor ordinal, we obtain $\colim_{\beta < \mu}
	X_\beta = X_\nu \in \cN$.  As $X_\mu$ is a pushout of $X_\nu$ along a
	generating cofibration, $X_\mu \in \cN$ by (ii), a contradiction.

	If $\mu$ is a limit ordinal, we immediately obtain $\colim_{\beta < \mu}
	X_\beta = X_\mu$ because $X$ is a $\lambda$-sequence. Therefore, $X_\mu
	\in \cN$, a contradiction as well.  
	In conclusion, $\colim_{\beta < \lambda} X_\beta \in \cN$ by (iii), and we
	have shown that $\cN$ contains all cofibrant objects.

	For the second claim, we need to check that the class $\cM^{\cof}$ of
	cofibrant objects of $\cM$ satisfies (i)-(iii). This is obvious for (i)
	and (ii). If $X$ is a $\lambda$-sequence such that (iii).(a)-(c) hold
	(note that (b) actually follows from (a) and (c) in this case), then the
	map $\emptyset = X_0 \to \colim_{\beta < \lambda} X_\beta$ is a
	cofibration, and the object $\colim_{\beta < \lambda} X_\beta$ is
	cofibrant.
\end{proof}